\documentclass[10pt]{amsart}
\usepackage{amsfonts}
\usepackage{ifthen}
\usepackage{amsthm}
\usepackage{amsmath}
\usepackage{graphicx}
\usepackage{amscd,amssymb,amsthm}

\newcounter{minutes}
\divide\time by 60
\newcounter{hours}
\multiply\time by 60 \addtocounter{minutes}{-\time}

\newtheorem{theorem}{Theorem}

\keywords{Bessel function of the first kind; zeros of Bessel functions of the first kind; infinite series; Taylor series coefficients; Mittag-Leffler expansion; Riccati differential equation; recursion; recursive algorithm.} \subjclass[2010]{33C10.}

\title[Riccati type recursions for some infinite series involving zeros of Bessel functions]{Riccati type recursions for some infinite series involving zeros of Bessel functions of the first kind}

\author[\'A. Baricz]{\'Arp\'ad Baricz}
\address{Department of Economics, Babe\c{s}-Bolyai University, 400591 Cluj-Napoca, Romania}
\address{Institute of Applied Mathematics, \'Obuda University, 1034 Budapest, Hungary}
\email{bariczocsi@yahoo.com}

\author[A.F. Skorka]{Antal Ferenc Skorka}
\address{Institute of Applied Mathematics, \'Obuda University, 1034 Budapest, Hungary}
\email{skorka.antal.ferenc@gmail.com}

\begin{document}
	
\def\thefootnote{}
\footnotetext{ \texttt{File:~\jobname .tex,
printed: \number\year-0\number\month-\number\day,
\thehours.\ifnum\theminutes<10{0}\fi\theminutes}}
\makeatletter\def\thefootnote{\@arabic\c@footnote}\makeatother
	
\maketitle

\begin{abstract}
Some infinite series involving the positive zeros of Bessel functions of the first kind are investigated. The motivation behind these series lies in quantum mechanical perturbation problems, in which energies and matrix elements of unperturbed states are expressible in terms of zeros of Bessel functions of the first kind. The existing approach by Pedersen and Urbanowicz for calculating these series involves the Thomas-Reiche-Kuhn sum rule, differential recurrences for powers of Bessel function ratios or application of Lommel polynomials. In this paper an alternative approach is provided: a recursive algorithm is proposed that theoretically can produce the infinite series values in question. Our method is relatively simple and rely on three main ingredients: the Mittag-Leffler expansion and Riccati differential equation for the quotient of Bessel functions of the first kind, as well as the Taylor series coefficients of these ratios. The technique employed in the paper could be useful to treat similar problems where infinite series of zeros of special functions is involved.
\end{abstract}

\section{\bf Evaluation of some infinite series involving zeros of Bessel functions}

Let us consider the infinite series constructed from the zeros of the Bessel function of the first kind of the following form
$$\mathcal{S}_r(p,n):=\sum_{m\geq1} \frac{j_{p,n}^2 j_{p+1,m}^2}{(j_{p+1,m}^2 - j_{p,n}^2)^r},$$
where $j_{p,n}$ denotes the $n$th positive zero of the Bessel function of the first kind $J_{p}$ of order $p>-1$ and $r$ is a natural number. This infinite series has been investigated recently by Pedersen \cite{Pedersen2018} in the following particular cases
$$\sum_{m\geq1}\frac{j_{p,n}^{2}j_{p+1,m}^{2}}{\left(j_{p+1,m}^{2}-j_{p,n}^{2}\right)^{3}}=\frac{p+1}{8},$$
$$\sum_{m\geq1}\frac{j_{p,n}^{2}j_{p+1,m}^{2}}{\left(j_{p+1,m}^{2}-j_{p,n}^{2}\right)^{5}}=\frac{2(p-2)(p^{2}-1)+(4p+1)j_{p,n}^{2}}{192\,j_{p,n}^{4}},$$
$$\sum_{m\geq1}\frac{j_{p,n}^{2}j_{p+1,m}^{2}}{\left(j_{p+1,m}^{2}-j_{p,n}^{2}\right)^{7}}=\frac{4(p-4)(p-3)(p-2)(p^{2}-1)+
2\left(16+p-49p^{2}+26p^{3}\right)j_{p,n}^{2}+(34p-1)j_{p,n}^{4}}{11520\,j_{p,n}^{8}}.$$
The motivation behind these series lies in quantum mechanical perturbation problems, since Bessel functions directly describe quantum states in spherical and cylindrical geometries. Under perfect confinement, the energies of Schr\"odinger particles are determined by zeros of Bessel functions of the first kind, while those of massless Dirac fermions are determined by intersections of consecutive zeros of Bessel functions of the first kind. When an external electric field is applied, perturbation theory expresses the polarizability as a sum over these zeros or intersections. For the models considered here, the polarizabilities can also be obtained analytically. Comparing the exact expressions with the perturbative sum-over-states formulas then produces rigorous sum rules. Thus, the above mentioned mathematical results are motivated by concrete quantum mechanical perturbation problems. For more details we refer to the paper \cite{Pedersen2018}.

In the present paper we are going first to investigate the infinite series $\mathcal{S}_r(p,n)$ considered by Pedersen. We note that the series $\mathcal{S}_3(p,n),$ $\mathcal{S}_5(p,n)$ and $\mathcal{S}_7(p,n)$ were deduced by Pedersen \cite{Pedersen2018} by using the so-called Thomas–Reiche–Kuhn sum rule, which states that the sum of the oscillator strengths for all possible electric dipole transitions from a given initial state is equal to the total number of electrons in the system. Recently, Urbanowicz \cite{Urbanowicz2026} proved that the series considered by Pedersen can be expressed via Calogero type sums by using the elementary relation
\begin{equation}\label{eqUrb}\sum_{m\geq1}\frac{j_{p,n}^2{j_{p+1,m}^2}}{\left({j_{p+1,m}^2}-j_{p,n}^2\right)^r}=j_{p,n}^2\sum_{m\geq1}\frac{1}{\left({j_{p+1,m}^2}-j_{p,n}^2\right)^{r-1}}
+j_{p,n}^4\sum_{m\geq1}\frac{1}{\left({j_{p+1,m}^2}-j_{p,n}^2\right)^r},\end{equation}
and the above Calogero type sums can be evaluated with the help of some differential recurrences for powers of Bessel function ratios. Moreover, Urbanowicz \cite{Urbanowicz2026} in particular evaluated the series $\mathcal{S}_2(p,n)$ and $\mathcal{S}_4(p,n).$ For more details on these series we refer to the recent papers \cite{fattah}, \cite{nowak}, \cite{Urbanowicz2023}, \cite{Urbanowicz2026} and to the references therein.

Now, we are going first to present a novel elementary alternative approach to evaluate the series $\mathcal{S}_r(p,n)$ for arbitrary $r$ and we will show later that this idea is working well also in the case of Calogero type sums. As a starting point we use the Mittag-Leffler expansion for Bessel functions of the first kind
\[\frac{J_{p+1}(x)}{J_p(x)}=\sum_{m\geq1}\frac{2x}{j_{p,m}^{2}-x^{2}},\]
and the three-term recurrence relation
\begin{equation}\label{eqthreeterm}
J_{p-1}(x)+J_{p+1}(x)=\frac{2p}{x}J_{p}(x).
\end{equation}
By changing $p$ to $p+1$ in the original Mittag-Leffler expansion as well as in the above recurrence relation, the Mittag-Leffler expansion
can be rewritten as \cite[eq. (30)]{Urbanowicz2023}
\begin{equation}\label{eq:bessel-fraction}
\frac{J_p(x)}{xJ_{p+1}(x)}=\frac{2(p+1)}{x^2}-\sum_{m\geq1}\frac{2}{j_{p+1,m}^2-x^2}.
\end{equation}
Now, differentiating $s-1$ times the right-hand side of the next expression
\[A_p(x):= \frac{J_p(\sqrt{x})}{\sqrt{x}\,J_{p+1}(\sqrt{x})}=\frac{2(p+1)}{x}-\sum_{m\geq1}\frac{2}{j_{p+1,m}^2-x},\]
we arrive at
\begin{equation}\label{eq:sum-q-from-A}A_p^{(s-1)}(x)=2(p+1)(-1)^{s-1}(s-1)!\,x^{-s}-2(s-1)!\sum_{m\geq1}\frac{1}{({j_{p+1,m}^2}-x)^s},\end{equation}
where $s\in\mathbb{N}.$ Consequently, we obtain that
\[\sum_{m\geq1}\frac{1}{({j_{p+1,m}^2}-j_{p,n}^2)^s}=(p+1)(-1)^{s-1}\left(j_{p,n}^2\right)^{-s}-\frac{A_p^{(s-1)}\left(j_{p,n}^2\right)}{2(s-1)!}\]
and applying \eqref{eqUrb} and \eqref{eq:sum-q-from-A} (in the case when $s=r-1$ and $s=r$), we conclude that for all $r\in\{2,3,\ldots\}$
\begin{equation}\label{eq:master-derivative-allr}
\sum_{m\geq1}\frac{j_{p,n}^{2}j_{p+1,m}^{2}}{\left(j_{p+1,m}^{2}-j_{p,n}^{2}\right)^{r}}=
-\frac{j_{p,n}^2}{2(r-2)!}\,A_p^{(r-2)}(j_{p,n}^2)-\frac{j_{p,n}^4}{2(r-1)!}\,A_p^{(r-1)}(j_{p,n}^2).
\end{equation}
Therefore, if we know the derivatives of $A_p(x)$ of order $r-1$ and $r-2$, then we should be able to retrieve the infinite series from this. However, this falls short in practical terms and we need a method to produce the derivatives. What we can do to achieve this is that we derive a Riccati type ordinary differential equation for the function $A_p(x)$, expand it as a Taylor series, and recursively calculate its coefficients. Essentially, the infinite series for different powers are encoded into the Taylor series coefficients. By using the standard recurrence relations for Bessel functions of the first kind
\[J_p'(v)=\frac{p}{v}J_p(v)-J_{p+1}(v)\quad \mbox{and} \quad J_{p+1}'(v)=-\frac{p+1}{v}J_{p+1}(v)+J_{p}(v),\]
it is clear that $B_p(v)=J_p(v)/J_{p+1}(v)$ satisfies the Riccati type ordinary differential equation
\[B_p'(v)=\frac{2p+1}{v}B_p(v)-1-\left[B_p(v)\right]^2,\]
which implies that $A_p(x)=B_p(v)/v,$ where $x=v^2,$ satisfies the next Riccati type differential equation
\begin{equation}\label{eq:A-riccati}
2xA_p'(x)+x\left[A_p(x)\right]^2-2p\cdot A_p(x)+1=0,
\end{equation}
since we clearly have that
\[\frac{dA_p(x)}{dv}= \frac{B_p'(v)}{v} - \frac{B_p(v)}{v^2}=\frac{1}{v}\left[\frac{2p+1}{v}B_p(v) - 1 - \left[B_p(v)\right]^2\right] - \frac{B_p(v)}{v^2}= \frac{2p}{v}A_p(x) - \frac{1}{v} - v\left[A_p(x)\right]^2.\]
This will be the key connection to the recursion. Our main observation is that after we construct the derivatives of the function $x\mapsto A_p(x)$ defined earlier, we will be able to calculate all the infinite series for arbitrary $r.$ The construction is based on the above deduced Riccati differential equation, in which we consider the Taylor series expansion around the fixed zeros ${j^2_{p,n}}$. It is clear that for all $n\in\mathbb{N}$ we have $A_p\left({j^2_{p,n}}\right) = 0,$ by definition, and the trick here is that the coefficients of the Taylor series will encode the infinite series values. More precisely, since for each $n\in\mathbb{N}$ fixed we have $A_p\left({j^2_{p,n}}\right)=0,$ the Taylor expansion of $x\mapsto A_p(x)$ about \(x={j^2_{p,n}}\) in a small neighborhood has the form
\[A_p\left(t+j^2_{p,n}\right)=\sum_{k\geq1}\alpha_k t^k,\qquad \alpha_k=\frac{A_p^{(k)}\left({j^2_{p,n}}\right)}{k!},\]
and we set $\alpha_0=A_p\left({j^2_{p,n}}\right)=0.$ Substituting $x=t+j^2_{p,n}$ into the Riccati differential equation \eqref{eq:A-riccati}, we arrive at
\begin{equation}
2(t+j^2_{p,n})A_p'\left(t+j^2_{p,n}\right)+(t+j^2_{p,n})\left[A_p\left(t+j^2_{p,n}\right)\right]^2-2p\cdot A_p\left(t+j^2_{p,n}\right)+1=0.\label{eq:ric-taylor}
\end{equation}
On the other hand
\[A_p'\left(t+j^2_{p,n}\right)=\sum_{k\geq1}k \alpha_k t^{k-1},\]
while using the Cauchy product formula, a simple consideration yields that
\[\left[A_p\left(t+j^2_{p,n}\right)\right]^2=\left[\sum_{k\geq1} \alpha_k t^k\right]^2=\sum_{k\geq2}\left[\sum_{j=1}^{k-1} \alpha_j \alpha_{k-j}\right]t^k.\]
Comparing the constant terms in \eqref{eq:ric-taylor} we obtain that $2j^2_{p,n}\alpha_1+1=0,$ that is, $\alpha_1=-1/\left(2j^2_{p,n}\right).$ Now, let \(k\ge 2\). The coefficient of \(t^k\) in \(2\left(t+j^2_{p,n}\right)A_p'\left(t+j^2_{p,n}\right)\) is
\[2j^2_{p,n}(k+1)\alpha_{k+1}+2k \alpha_k,\]
the coefficient of \(t^k\) in \(\left(t+j^2_{p,n}\right)\left[A_p\left(t+j^2_{p,n}\right)\right]^2\) is
\[j^2_{p,n}\sum_{j=1}^{k-1} \alpha_j \alpha_{k-j}+\sum_{j=1}^{k-2} \alpha_j \alpha_{k-1-j},\]
and the coefficient of \(t^k\) in \(-2p\cdot A_p\left(t+j^2_{p,n}\right)\) is \(-2p\cdot \alpha_k\). Therefore, the comparison of the coefficients in \eqref{eq:ric-taylor} yields
\[2j^2_{p,n}(k+1)\alpha_{k+1}+2(k-p)\alpha_k+j^2_{p,n}\sum_{j=1}^{k-1} \alpha_j \alpha_{k-j}+\sum_{j=1}^{k-2} \alpha_j \alpha_{k-1-j}=0,\]
or equivalently
\begin{equation}\label{eq:cn-recurrence-solved}
\alpha_{k+1}=-\frac{1}{2j^2_{p,n}(k+1)}\left[2(k-p)\alpha_k+j^2_{p,n}\sum_{j=1}^{k-1} \alpha_j \alpha_{k-j}+\sum_{j=1}^{k-2} \alpha_j \alpha_{k-1-j}\right], \quad k\geq2,
\end{equation}
with initial values $\alpha_0=0$ and $\alpha_1=-1/\left(2j^2_{p,n}\right).$ This recursion determines all coefficients $\alpha_k$ for $k\in\mathbb{N}.$ Now, recall that $\alpha_k=A_p^{(k)}({j_p,_n}^2)/k!,$ and consequently the identity \eqref{eq:master-derivative-allr} can be rewritten, for every integer $r\geq 2,$ as
\[\mathcal{S}_r(p,n)=\sum_{m\geq1}\frac{j_{p,n}^{2}j_{p+1,m}^{2}}{\left(j_{p+1,m}^{2}-j_{p,n}^{2}\right)^{r}}=-\frac12\left(j_{p,n}^2\alpha_{r-2}+j_{p,n}^4\alpha_{r-1}\right).\]
Accordingly, for a fixed $r\geq 2$ integer we can compute $\alpha_2,\alpha_3,\ldots,\alpha_{r-1}$ and then we can substitute the resulting coefficients into the above relation. Based on the recursion formula for $\alpha_{r+1}$ the algorithm runs in $\mathcal{O}(r^2)$ time provided that the $\alpha_r$ values are purely numerical, however the generation of these formulas requires symbolic calculation, and assuming that every $\alpha_r$ is already some complicated polynomial-like function, the complexity could be even $\mathcal{O}(r^3)$ or worse. However, various efficiency improvement practices can be applied to retrieve a high number of series values in a relatively short time. Although the optimization of the algorithm is out of scope of this paper, we note that the calculation of the coefficients can be made more efficient by FFT. The program that we used for generating the series is a simple Python script that recursively calculates the $\alpha_r$ coefficients, and assembles the $\mathcal{S}_r(p,n)$ final result. It utilizes some built-in packages (sympy) that support symbolic calculation. It has been observed that the program was able to exactly reproduce the existing results for $r \in\{1,2,3,4,5,6,7\}.$ The authors are open to provide access to the source code and running instructions if requested. Now, assuming the existing results, our first new formula is for $r=8$. For this we refer to Apendix A, where more evaluations can be found for  $r\in\{2,3,\ldots, 15\}$. Although the general closed formula is not yet known, due to the relatively simple recursive formula \eqref{eq:cn-recurrence-solved} it is still possible to deduce some basic structure. More precisely, it is clear that for $r\geq 3$ integer
\[\mathcal{S}_r(p,n)=\frac{P_r\!\left(p,\,j_{p,n}^{2}\right)}{2^{r}(r-1)!\,\left(j_{p,n}^{2}\right)^{r-3}},\]
where $P_r$ is a two-variable polynomial in $p$ of degree $r-2$ and in $j_{p,n}^2$ of degree at most $r-2$. By induction on the recurrence \eqref{eq:cn-recurrence-solved} for the coefficients $\alpha_{k+1}$ it is clear that every step adds exactly one power of $p$. The next result shows the basic structure of the coefficients in \eqref{eq:cn-recurrence-solved}. In what follows for a real number $\tau$ the expression $\left\lfloor \tau\right\rfloor$ means $\max\left\{\left.m\in\mathbb{Z}\right|m\leq\tau\right\},$ that is the greatest integer less than or equal to $\tau.$
	
\begin{theorem}
For every $n\geq1$ fixed and $k\geq1$ the coefficients in \eqref{eq:cn-recurrence-solved} can be written as
\[\alpha_k=\frac{R_k\left(p,\,j_{p,n}^{2}\right)}{2^k k!\,\left(j_{p,n}^{2}\right)^k}, \qquad \deg_{j_{p,n}^{2}} R_k\left(p,\,j_{p,n}^{2}\right)=\left\lfloor \frac{k-1}{2}\right\rfloor,\]
where $R_k$ is a polynomial. Consequently, for every integer $r\geq3$
\[\mathcal{S}_r(p,n)=\frac{P_r\left(p,\,j_{p,n}^{2}\right)}{2^r (r-1)!\,{\left(j_{p,n}^{2}\right)}^{\,r-3}},\]
where $P_r$ is a polynomial satisfying
\[\deg_{j_{p,n}^{2}} P_r\left(p,\,j_{p,n}^{2}\right)=\left\lfloor \frac{r-2}{2}\right\rfloor\]
or equivalently
\[\deg_{j_{p,n}}\!\left(\mbox{numerator of}\ \mathcal{S}_r\right)=2\left\lfloor \frac{r-2}{2}\right\rfloor=
\begin{cases}
r-2,& r\ \text{is even}\\
r-3,& r\ \text{is odd}
\end{cases}.\]
\end{theorem}

\begin{proof}
We prove by induction on $k$ that
\[\alpha_k=\frac{R_k\left(p,\,j_{p,n}^{2}\right)}{2^k k!\,\left(j_{p,n}^{2}\right)^k}, \qquad \deg_{j_{p,n}^{2}} R_k=\left\lfloor \frac{k-1}{2}\right\rfloor.\]
For $k=1$ we clearly have that $\alpha_1=-1/\left(2j_{p,n}^{2}\right)=-1/\left(2^1 1!\,j_{p,n}^{2}\right),$ and thus in this case $R_1=-1$ and the claim holds. Now, assume that
for all $j\in\{1,2,\ldots,k\}$ we have that	$\alpha_j=R_j\left(p,j_{p,n}^{2}\right)/\left(2^j j!\,j_{p,n}^{2j}\right).$ Substituting this into the recursion, we first rewrite every term over the common denominator $2^k k!\,j_{p,n}^{2k}.$ Indeed, by using the shorthand notation $R_j=R_j\left(p,j_{p,n}^{2}\right),$ we arrive at
\[2(k-p)\alpha_k=\frac{2(k-p)R_k}{2^k k!\,j_{p,n}^{2k}},\]
\[j_{p,n}^{2}\,\alpha_j\alpha_{k-j}=\frac{R_jR_{k-j}}{2^k j!(k-j)!\,j_{p,n}^{2k-2}}=\frac{\binom{k}{j}\,j_{p,n}^{2}\,R_jR_{k-j}}{2^k k!\,j_{p,n}^{2k}},\]
and
\[\alpha_j\alpha_{k-1-j}=\frac{R_jR_{k-1-j}}{2^{k-1}j!(k-1-j)!\,j_{p,n}^{2k-2}}=\frac{2k\binom{k-1}{j}\,j_{p,n}^{2}\,R_jR_{k-1-j}}{2^k k!\,j_{p,n}^{2k}}.\]
Therefore the whole numerator in the recursion has the form
\[\frac{1}{{2^k k!\,j_{p,n}^{2k}}}\left[2(k-p)R_k+j_{p,n}^{2}\sum_{j=1}^{k-1}\binom{k}{j}R_jR_{k-j}+2k\,j_{p,n}^{2}\,\sum_{j=1}^{k-2}\binom{k-1}{j}R_jR_{k-1-j}\right].\]
Now, dividing the above expression by $2j_{p,n}^{2}(k+1),$ we clearly obtain that
\[\alpha_{k+1}=-\frac{1}{2^{k+1}(k+1)!\,j_{p,n}^{2k+2}}\left[2(k-p)R_k+j_{p,n}^{2}\sum_{j=1}^{k-1}\binom{k}{j}R_jR_{k-j}+2kj_{p,n}^{2}\sum_{j=1}^{k-2}\binom{k-1}{j}R_jR_{k-1-j}\right]\]
meaning that
\[R_{k+1}=-\left[2(k-p)R_k+j_{p,n}^{2}\sum_{j=1}^{k-1}\binom{k}{j}R_jR_{k-j}+2kj_{p,n}^{2}\sum_{j=1}^{k-2}\binom{k-1}{j}R_jR_{k-1-j}\right]\]
is again a polynomial, and the denominator fulfils the induction step. Moreover, since $\left\lfloor \tau\right\rfloor+\left\lfloor \mu\right\rfloor\leq \left\lfloor \tau+\mu\right\rfloor$ for every $\tau$ and $\mu$ real, we observe that
\[\deg_{j_{p,n}^{2}}\!\left(2(k-p)R_k\right)\le \left\lfloor\frac{k-1}{2}\right\rfloor,\]
\[\deg_{j_{p,n}^{2}}\!\left(j_{p,n}^{2}R_jR_{k-j}\right)=1+\left\lfloor\frac{j-1}{2}\right\rfloor+\left\lfloor\frac{k-j-1}{2}\right\rfloor\le\left\lfloor\frac k2\right\rfloor,\]
and similarly
\[\deg_{j_{p,n}^{2}}\!\left(j_{p,n}^{2}R_jR_{k-1-j}\right)\le \left\lfloor\frac{k-1}{2}\right\rfloor.\]
Consequently we arrive at
\[\deg_{j_{p,n}^{2}} R_{k+1}\le \left\lfloor\frac{3k-2}{2}\right\rfloor=k-1+\left\lfloor\frac k2\right\rfloor\leq \left\lfloor\frac k2\right\rfloor.\]
This bound is attained: if \(k=2m\), and we choose \(j\) odd in the middle sum, giving degree
\[1+\frac{j-1}{2}+\frac{2m-j-1}{2}=m,\]
and	if \(k=2m+1\), the first term \(2(k-p)R_k\) already has degree \(m\). Therefore
\[\deg_{j_{p,n}^{2}} R_{k+1}=\left\lfloor\frac k2\right\rfloor,\]
and the induction is complete. Moreover, it is clear that
\[\mathcal{S}_r(p,n)=-\frac12\left[j_{p,n}^{2}\frac{R_{r-2}\left(p,j_{p,n}^{2}\right)}{2^{r-2}(r-2)!\,j_{p,n}^{2r-4}}+j_{p,n}^{4}\frac{R_{r-1}\left(p,j_{p,n}^{2}\right)}
{2^{r-1}(r-1)!\,j_{p,n}^{2r-2}}\right],\]
and thus we have
\[\mathcal{S}_r(p,n)=\frac{P_r\left(p,j_{p,n}^{2}\right)}{2^r (r-1)!\,j_{p,n}^{2r-6}},\]
where
\[-P_r\left(p,j_{p,n}^{2}\right)=2(r-1)R_{r-2}\left(p,j_{p,n}^{2}\right)+R_{r-1}\left(p,j_{p,n}^{2}\right).\]
This in turn implies that
\[\deg_{j_{p,n}^{2}} P_r=\max\!\left\{\deg_{j_{p,n}^{2}} R_{r-2},\deg_{j_{p,n}^{2}} R_{r-1}\right\}=
\max\!\left\{\left\lfloor\frac{r-3}{2}\right\rfloor,\left\lfloor\frac{r-2}{2}\right\rfloor\right\}=\left\lfloor\frac{r-2}{2}\right\rfloor\]
and the numerator degree in \(j_{p,n}\) is indeed
\[2\left\lfloor\frac{r-2}{2}\right\rfloor=
\begin{cases}
r-2,& r\ \text{is even}\\
r-3,& r\ \text{is odd}
\end{cases}.\]
\end{proof}

\section{\bf Evaluation of some Calogero type infinite series involving zeros of Bessel functions}

Now, for $r,n\in\mathbb{N}$ let us consider the Calogero type infinite series
\[\mathcal T_r^{+}(p,n):=\sum_{m\ge1}\frac{1}{\left(j_{p,m}^{2}-j_{p+1,n}^{2}\right)^r},\quad p>-1\]
and
\[\mathcal T_r^{-}(p,n):=\sum_{m\ge1}\frac{1}{\left(j_{p,m}^{2}-j_{p-1,n}^{2}\right)^r},\quad p>0.\]
The first family of series is centered at a zero of $J_{p+1}$, while the second is centered at a zero of $J_{p-1}$. We note that by using some new differential recurrences for powers of Bessel function ratios, quite recently Urbanowicz evaluated for $r\in\{1,2,3,4,5,6\}$ the series $\mathcal T_r^{+}(p,n)$ (see \cite[Theorem 1]{Urbanowicz2026}), as well as for $r\in\{1,2,3,4,5,6\}$ the series $\mathcal T_r^{-}(p,n)$ (see \cite[Theorem 2]{Urbanowicz2026}). Moreover, both series were already studied by the same author by using a sophisticated analysis involving Lommel polynomials in the case when $r\in\{1,2,3\},$ see \cite{Urbanowicz2023} for more details. In what follows our aim is to show that with some slight modifications our idea, presented in the previous section, is working also well in the case of these Calogero type infinite series involving zeros of Bessel functions of the first kind. The idea here is to consider in this case the function
\begin{equation}\label{eq:W-def}
D_p(x):=\frac{J_{p+1}(\sqrt{x})}{2\sqrt{x}\,J_p(\sqrt{x})}=\sum_{m\ge1}\frac{1}{j_{p,m}^{2}-x}.
\end{equation}
The centers used below are not poles of $D_p$, since adjacent-order Bessel functions do not have common positive zeros, that is, their zeros have the so-called interlacing property, which is well-known in the literature. Thus, the following Taylor series expansions are local expansions of a meromorphic function at regular points.
More explicitly, we have
\begin{equation}\label{eq:upper-taylor}
D_p\!\left(t+j_{p+1,n}^{2}\right)
=
\sum_{k\ge0}\beta_k^{+}t^k,
\qquad
\beta_k^{+}=\mathcal T_{k+1}^{+}(p,n)=\frac{D_p^{(k)}\left(j_{p+1,n}^2\right)}{k!},
\end{equation}
and
\begin{equation}\label{eq:lower-taylor}
D_p\!\left(t+j_{p-1,n}^{2}\right)
=
\sum_{k\ge0}\beta_k^{-}t^k,
\qquad
\beta_k^{-}=\mathcal T_{k+1}^{-}(p,n)=\frac{D_p^{(k)}\left(j_{p-1,n}^2\right)}{k!}.
\end{equation}
Indeed, for every $s\in\mathbb N$ and for every such center $\xi\neq j_{p,m}^2$,
differentiating \eqref{eq:W-def} $s-1$ times gives
\begin{equation}\label{eq:derivative-series}
\frac{D_p^{(s-1)}(\xi)}{(s-1)!}
=
\sum_{m\ge1}\frac{1}{\left(j_{p,m}^{2}-\xi\right)^s}.
\end{equation}
Taking $\xi=j_{p+1,n}^2$ or $\xi=j_{p-1,n}^2$ gives \eqref{eq:upper-taylor} and \eqref{eq:lower-taylor}, respectively. The preceding observation shows that the series values are encoded in the derivatives of $D_p$ at the neighboring squared zeros. In principle, repeatedly differentiating the Mittag--Leffler expansion already produces the answer through \eqref{eq:derivative-series}.
For actual computation, however, it is more efficient to compute the Taylor series coefficients of $D_p$ recursively. By using the notations
\[
C_p(v):=\frac{1}{B_p(v)}=\frac{J_{p+1}(v)}{J_p(v)},
\qquad
D_p(x)=\frac{C_p(v)}{2v},
\qquad x=v^2,
\]
in view of the standard recurrence relations
\[
J_p'(v)=\frac{p}{v}J_p(v)-J_{p+1}(v),
\qquad
J_{p+1}'(v)=J_p(s)-\frac{p+1}{v}J_{p+1}(v),
\]
we obtain
\[
C_p'(v)
=
\frac{J_{p+1}'(v)J_p(v)-J_{p+1}(v)J_p'(v)}{J_p(v)^2}
=
1-\frac{2p+1}{v}C_p(v)+C_p^2(v).
\]
Since $C_p(v)=2vD_p(v^2)$, differentiation gives
\[
C_p'(v)=2D_p(x)+4xD_p'(x).
\]
Substituting $C_p(v)=2vD_p(x)$ into the preceding Riccati differential equation and using $x=v^2$ yields
\begin{equation}\label{eq:W-riccati}
xD_p'(x)-xD_p^2(x)+(p+1)D_p(x)-\frac14=0.
\end{equation}
This ordinary Riccati differential equation is the connection between the above infinite sums and a finite recursion for the Taylor series coefficients. Now, we fix a center $\zeta$ equal either to $j_{p+1,n}^2$ or to $j_{p-1,n}^2$, and write
\[
D_p(t+\zeta)=\sum_{k\ge0}\beta_k t^k.
\]
Substitution of $x=t+\zeta$ into \eqref{eq:W-riccati} gives
\begin{equation}\label{eq:W-riccati-taylor}
(t+\zeta)D_p'(t+\zeta)-(t+\zeta)\left[D_p(t+\zeta)\right]^2+(p+1)D_p(t+\zeta)-\frac14=0.
\end{equation}
On the other hand
\[
D_p'(t+\zeta)=\sum_{k\ge0}(k+1)\beta_{k+1}t^k,
\]
and by the Cauchy product formula we clearly have
\[
\left[D_p(t+\zeta)\right]^2
=
\sum_{k\ge0}\left[\sum_{j=0}^{k}\beta_j\beta_{k-j}\right]t^k.
\]
Comparing the constant terms in \eqref{eq:W-riccati-taylor} gives
\[
\zeta \beta_1-\zeta \beta_0^2+(p+1)\beta_0-\frac14=0.
\]
For $k\ge1$, the coefficient of $t^k$ in $(t+\zeta)D_p'(t+\zeta)$ is $\zeta(k+1)\beta_{k+1}+k\beta_k,$ while
the coefficient of $t^k$ in $(t+\zeta)[D_p(t+\zeta)]^2$ is
\[
\zeta\sum_{j=0}^{k}\beta_j\beta_{k-j}
+
\sum_{j=0}^{k-1}\beta_j\beta_{k-1-j}.
\]
Moreover, the coefficient of $t^k$ in $(p+1)D_p(t+\zeta)$ is $(p+1)\beta_k$. Therefore, we arrive at
\begin{equation}\label{eq:b-recursion-general}
\beta_{k+1}
=\frac{1}{\zeta(k+1)}\left[\zeta\sum_{j=0}^{k}\beta_j\beta_{k-j}
+\sum_{j=0}^{k-1}\beta_j\beta_{k-1-j}-(k+p+1)\beta_k\right],
\qquad k\ge1.
\end{equation}
The same formula also holds for $k=0$ if the second sum is interpreted as zero and the term $1/4$ is added to the numerator. For the polynomial structure it is useful to normalize the coefficients by $\gamma_k:=\zeta^{k+1}\beta_k.$ Equivalently, if $u=t/\zeta$ and
\[
\gamma_p(u):=\sum_{k\ge0}\gamma_k u^k,
\]
then $D_p(t+\zeta)=\zeta^{-1}\gamma_p(u)$.
After multiplying \eqref{eq:W-riccati-taylor} by $\zeta$, the equation becomes
\[
(1+u)\gamma_p'(u)-(1+u)\left[\gamma_p(u)\right]^2+(p+1)\gamma_p(u)-\frac{\zeta}{4}=0.
\]
Comparing the coefficient of $u^k$ gives the normalized recursion
\begin{equation}\label{eq:c-recursion-general}
\gamma_{k+1}
=\frac{1}{k+1}\left[\sum_{j=0}^{k}\gamma_j\gamma_{k-j}+\sum_{j=0}^{k-1}\gamma_j\gamma_{k-1-j}
-(k+p+1)\gamma_k+\frac{\zeta}{4}\delta_{k0}\right],
\qquad k\ge0,
\end{equation}
where the second sum is interpreted as zero for $k=0$. In the upper case $\zeta=j_{p+1,n}^{2}$ and we use the notation $\gamma_k^{+}:=\left(j_{p+1,n}^{2}\right)^{k+1}\beta_k^{+},$ while in the lower case $\zeta=j_{p-1,n}^{2}$ and we write $\gamma_k^{-}:=\left(j_{p-1,n}^{2}\right)^{k+1}\beta_k^{-}.$ Thus, for both signs and every $r\ge1$ we have
$\mathcal T_r^{\pm}(p,n)={\gamma_{r-1}^{\pm}}/{\zeta_\pm^r},$ where $\zeta_+=j_{p+1,n}^{2}$ and $\zeta_-=j_{p-1,n}^{2},$ with the corresponding initial values $\gamma_0^{+}=0$ and $\gamma_0^{-}=p.$

Since $J_{p+1}(j_{p+1,n})=0$, equation \eqref{eq:W-def} gives
$\mathcal T_1^{+}(p,n)=D_p\!\left(j_{p+1,n}^{2}\right)=0.$ Thus $\gamma_0^{+}=0$. Now, we define the two-variable polynomials $E_m(p,y)$ by $E_1(p,y)=\frac14$ and for $k\ge2$ as follows
$$E_{k+1}(p,y)=\frac{1}{k+1}\left[y\sum_{j=1}^{k-1}E_j(p,y)E_{k-j}(p,y)+y\sum_{j=1}^{k-2}E_j(p,y)E_{k-1-j}(p,y)
-(k+p+1)E_k(p,y)\right],$$
where empty sums are interpreted as zero. We note that the exact values of the series $\mathcal T_r^{+}(p,n)$ in the case of $r\in\{1,2,\ldots,10\}$ can be found in Appendix B.

\begin{theorem}\label{thm:upper-structure}
For every $k\ge1$ we have $\gamma_k^{+}=j_{p+1,n}^{2}E_k\!\left(p,j_{p+1,n}^{2}\right).$ Moreover, the degrees of the polynomial are
\[\deg_y E_k(p,y)=\left\lfloor\frac{k-1}{2}\right\rfloor, \quad \deg_p E_k(p,y)=k-1\]
and the leading coefficient of $p^{k-1}$ is exactly ${(-1)^{k-1}}/{(4k!)}.$ Consequently, for every integer $r\ge2$
\begin{equation}\label{eq:T-plus-structure}
\mathcal T_r^{+}(p,n)=\frac{E_{r-1}\!\left(p,j_{p+1,n}^{2}\right)}{\left(j_{p+1,n}^{2}\right)^{r-1}},
\end{equation}
and the numerator in \eqref{eq:T-plus-structure} satisfies
\[\deg_{j_{p+1,n}^{2}}\!\left(\mbox{numerator of}\ \ \mathcal T_r^{+}\right)=\left\lfloor\frac{r-2}{2}\right\rfloor, \quad \deg_p\!\left(\mbox{numerator of}\ \ \mathcal T_r^{+}\right)=r-2,\]
or equivalently
\[\deg_{j_{p+1,n}}\!\left(\mbox{numerator of}\ \ \mathcal T_r^{+}\right)=2\left\lfloor\frac{r-2}{2}\right\rfloor
=
\begin{cases}
r-2,& r\ \text{is even},\\
r-3,& r\ \text{is odd}.
\end{cases}\]
\end{theorem}

\begin{proof}
Recall that $\gamma_0^{+}=0.$ By using \eqref{eq:c-recursion-general} with $\zeta=j_{p+1,n}^{2}$ gives $\gamma_1^{+}={j_{p+1,n}^{2}}/{4}.$ Now, assume that $\gamma_j^{+}=j_{p+1,n}^{2}E_j(p,j_{p+1,n}^{2})$ for $1\le j\le k.$ Since $\gamma_0^{+}=0$, all convolution terms containing $\gamma_0^{+}$ vanish in \eqref{eq:c-recursion-general}, and substitution gives
\[
\gamma_{k+1}^{+}
=\frac{j_{p+1,n}^{2}}{k+1}\left[
j_{p+1,n}^{2}\sum_{j=1}^{k-1}E_jE_{k-j}
+
j_{p+1,n}^{2}\sum_{j=1}^{k-2}E_jE_{k-1-j}
-(k+p+1)E_k\right]=j_{p+1,n}^{2}E_{k+1},
\]
where each $E_j$ is evaluated at $(p,j_{p+1,n}^{2}).$ We next prove the degree statements for $E_k.$ The $p$-degree is immediate: the convolution terms have $p$-degree at most $k-2$, while the only contribution to the top $p$-degree of $E_{k+1}$ comes from $-pE_k/(k+1).$ Thus, if the coefficient of the term $p^{k-1}$ in $E_k(p,y)$ is denoted by $\lambda_k,$ then clearly $\lambda_{k+1}=-{\lambda_k}/{(k+1)}$ with $\lambda_1=1/4$ and therefore $\lambda_k={(-1)^{k-1}}/{(4k!)}$ for each $k\geq1$ integer. For the degree in $y$, mathematical induction first gives the upper bound
\[
\deg_yE_k(p,y)\le\left\lfloor\frac{k-1}{2}\right\rfloor=:d_k.
\]
To prove that this bound is attained we denote the coefficient of $y^{d_k}$ in $E_k(p,y)$ simply as $\omega_k$ and split the leading coefficients into
$a_s:=\omega_{2s+1}$ and $b_s:=\omega_{2s+2}.$ For the odd subsequence, $a_0=1/4$ and the top $y$-degree part gives
$$
a_s=\frac{1}{2s+1}\sum_{j=0}^{s-1}a_ja_{s-1-j},
\qquad s\ge1.
$$
Hence $a_s>0$ for every $s\ge0$. For the even subsequence, let $\mu_s$ be the coefficient of $p$ in $b_s$. Since $E_2(p,y)=-({p+2})/{8}$ we have $\mu_0=-1/8,$ and taking the coefficient of $p$ in the top $y$-degree part gives
$$
\mu_s
=
\frac{1}{2s+2}\left[2\sum_{j=0}^{s-1}a_j\mu_{s-1-j}-a_s\right],
\qquad s\ge1.
$$
Because every $a_j>0$ and $\mu_0<0$, mathematical induction implies that $\mu_s<0$ for every $s\ge0$. This in turn implies that $b_s\ne0$ for every $s\ge0$. The top $y$-degree coefficient is therefore nonzero in both the odd and the even subsequences. Finally, observe that
\[\mathcal T_r^{+}(p,n)
=\beta_{r-1}^{+}
=\frac{\gamma_{r-1}^{+}}{\left(j_{p+1,n}^{2}\right)^r}
=\frac{E_{r-1}(p,j_{p+1,n}^{2})}{\left(j_{p+1,n}^{2}\right)^{r-1}}.\]
\end{proof}

Since $J_{p-1}(j_{p-1,n})=0$, the recurrence relation
\[J_{p+1}(x)=\frac{2p}{x}J_p(x)-J_{p-1}(x)\]
implies that
\[\frac{J_{p+1}(j_{p-1,n})}{J_p(j_{p-1,n})}=\frac{2p}{j_{p-1,n}}.\]
Therefore we clearly have
\begin{equation}\label{eq:T-minus-one}
\mathcal T_1^{-}(p,n)=D_p\!\left(j_{p-1,n}^{2}\right)=\frac{p}{j_{p-1,n}^{2}}.
\end{equation}
Now, define the two-variable polynomials as follows
\begin{equation}\label{eq:F-def-final}
F_k(p,y):=(-1)^k p+yE_k(-p,y),
\qquad k\ge1.
\end{equation}

The next theorem is the counterpart of Theorem \ref{thm:upper-structure}. We note that the exact values of the series $\mathcal T_r^{-}(p,n)$ in the case of $r\in\{1,2,\ldots,10\}$ can be found in Appendix C.

\begin{theorem}\label{thm:lower-structure}
For every $k\ge1$ we have $\gamma_k^{-}=F_k\!\left(p,j_{p-1,n}^{2}\right).$ Moreover, the degrees of the polynomial are
\[\deg_yF_k(p,y)=1+\left\lfloor\frac{k-1}{2}\right\rfloor
=\left\lfloor\frac{k+1}{2}\right\rfloor,\quad \deg_pF_k(p,y)=\max\{1,k-1\}.\]
Consequently, for every integer $r\ge2$,
\begin{equation}\label{eq:T-minus-structure}
\mathcal T_r^{-}(p,n)
=
\frac{F_{r-1}\!\left(p,j_{p-1,n}^{2}\right)}{\left(j_{p-1,n}^{2}\right)^r},
\end{equation}
and the numerator in \eqref{eq:T-minus-structure} satisfies
\[
\deg_{j_{p-1,n}^{2}}\!\left(\mbox{numerator of}\ \ \mathcal T_r^{-}\right)
=
\left\lfloor\frac r2\right\rfloor,\quad
\deg_p\!\left(\mbox{numerator of}\ \ \mathcal T_r^{-}\right)
=
\max\{1,r-2\},
\]
or equivalently
\[
\deg_{j_{p-1,n}}\!\left(\mbox{numerator of}\ \ \mathcal T_r^{-}\right)
=
2\left\lfloor\frac r2\right\rfloor
=
\begin{cases}
r,& r\ \text{is even},\\
r-1,& r\ \text{is odd}.
\end{cases}
\]
\end{theorem}

\begin{proof}
First observe that equation \eqref{eq:T-minus-one} implies $\gamma_0^{-}=p$. Now, if we choose $y=j_{p-1,n}^{2}$ and define $Q_0:=0$ and $\gamma_k^{-}=(-1)^kp+Q_k$ for all $k\in\mathbb{N},$ then our aim is to prove that $Q_k=yE_k(-p,y)$ for all $k\in\mathbb{N},$ which is exactly $\gamma_k^{-}=F_k\!\left(p,j_{p-1,n}^{2}\right).$ For $k=0$ the recursive relation \eqref{eq:c-recursion-general} implies that
\[
\gamma_1^{-}
=(\gamma_0^{-})^2-(p+1)\gamma_0^{-}+\frac{y}{4}
=p^2-(p+1)p+\frac{y}{4}
=-p+\frac{y}{4}.
\]
Hence $Q_1=y/4=yE_1(-p,y)$. Now, assume now that $Q_j=yE_j(-p,y)$ holds true for all $1\le j\le k$. Substitute $\gamma_j^{-}=(-1)^jp+Q_j$ into \eqref{eq:c-recursion-general}.
The part containing only the alternating terms $(-1)^jp$ is
\begin{align*}
\sum_{j=0}^{k}&(-1)^jp(-1)^{k-j}p+\sum_{j=0}^{k-1}(-1)^jp(-1)^{k-1-j}p-(k+p+1)(-1)^kp\\
&=(k+1)(-1)^kp^2+k(-1)^{k-1}p^2-(k+p+1)(-1)^kp=(k+1)(-1)^{k+1}p.
\end{align*}
After division by $k+1$ this gives the required alternating term $(-1)^{k+1}p$. It remains to identify the non-alternating part. The mixed terms in the two convolutions cancel in pairs except for the two boundary occurrences of $Q_k$: for $1\le \ell\le k-1$ the coefficient of $Q_\ell$ is exactly $2(-1)^{m-\ell}p+2(-1)^{m-1-\ell}p=0,$ whereas the coefficient of $Q_k$ is $2p$. Consequently the recurrence relation for $Q_k$ is the following
\[Q_{k+1}=\frac{1}{k+1}\left[\sum_{j=1}^{k-1}Q_jQ_{k-j}+\sum_{j=1}^{k-2}Q_jQ_{k-1-j}-(k-p+1)Q_k\right],\qquad k\ge1,\]
where the empty sums are zero. Substituting the induction hypothesis $Q_j=yE_j(-p,y)$ gives
\[
Q_{k+1}
=\frac{y}{k+1}\left[y\sum_{j=1}^{k-1}E_j(-p,y)E_{k-j}(-p,y)
+y\sum_{j=1}^{k-2}E_j(-p,y)E_{k-1-j}(-p,y)
-(k-p+1)E_k(-p,y)\right],
\]
which is precisely $yE_{m+1}(-p,y)$ and thus the mathematical induction is complete.

The degree in $y$ follows from \eqref{eq:F-def-final} and the previous theorem
\[
\deg_y\bigl(yE_k(-p,y)\bigr)
=1+\left\lfloor\frac{m-1}{2}\right\rfloor,
\]
and the additional term $(-1)^kp$ has $y$-degree zero. For the $p$-degree, if $m=1$, then $F_1(p,y)=-p+{y}/{4},$ and thus $\deg_pF_1=1$. If $m=2$, then $E_2(p,y)=-(p+2)/8$, and therefore
\[
F_2(p,y)=p+yE_2(-p,y)=\left(1+\frac y8\right)p-\frac y4,
\]
that is $\deg_pF_2=1$. If $k\ge3$, then the top $p$-degree term can only come from $yE_k(-p,y)$ which is nonzero. Finally, observe that
\[
\mathcal T_r^{-}(p,n)
=\beta_{r-1}^{-}
=
\frac{\gamma_{r-1}^{-}}{\left(j_{p-1,n}^{2}\right)^r}
=
\frac{F_{r-1}(p,j_{p-1,n}^{2})}{\left(j_{p-1,n}^{2}\right)^r},
\]
which completes the proof.
\end{proof}
	
\section{\bf Concluding remarks}

In this paper we presented an algorithmic way to express recursively the values of some infinite series involving zeros of Bessel functions of the first kind. The main idea was to consider the Ricatti ordinary differential equation of some quotient of Bessel functions of the first kind and to combine this with the Taylor series of the corresponding ratio. In this way we obtained a recursive formula for the coefficients of the Taylor series from which we were able to compute with a suitable Python script the exact expressions of the infinite series in question. In this section we would like to mention briefly that the approach used in the previous section may be useful for the evaluation of more general Calogero type infinite series. Namely, if for $q,r,n\in\mathbb{N}$ and suitable $p$ real we consider the series
\[\mathcal{T}_{r}^{(-q)}(p,n):=\sum_{m\geq1}\frac{1}{\left(j_{p,m}^{2}-j_{p-q,n}^{2}\right)^{r}}\quad \mbox{and} \quad
\mathcal{T}_{r}^{(+q)}(p,n):=\sum_{m\geq1}\frac{1}{\left(j_{p,m}^{2}-j_{p+q,n}^{2}\right)^{r}},\]
then by choosing in \eqref{eq:derivative-series} the values $\zeta=j_{p-q,n}^2$ or $\zeta=j_{p+q,n}^2,$ we arrive at
\[\mathcal{T}_{r}^{(\pm q)}(p,n):=\sum_{m\geq1}\frac{1}{\left(j_{p,m}^{2}-j_{p\pm q,n}^{2}\right)^{r}}=\frac{D_p^{(r-1)}\left(j_{p\pm q,n}^{2}\right)}{(r-1)!},\]
where we recall that
\[D_p(x)=\frac{J_{p+1}(\sqrt{x})}{2\sqrt{x}J_p(\sqrt{x})},\qquad D_p(t+\zeta)=\sum_{k\geq0}\beta_kt^k, \qquad \beta_k=\frac{D_p^{(k)}(\zeta)}{k!}\]
and $\beta_k$ satisfies the recurrence relation \eqref{eq:b-recursion-general}. Thus, once we know the initial value $\beta_0^{(\pm q)}=D_p\left(j_{p\pm q,n}^{2}\right),$ then the same approach is working also in this general case, that is the recursion \eqref{eq:b-recursion-general} generates every desired sum, with of course more complicated computations. Now, for the initial values we consider first for $k\in\{1,2,\ldots,q\}$ the expression
\[\rho_k(x)=\frac{J_{p-q+k+1}(x)}{J_{p-q+k}(x)},\]
which in view of the three-term recurrence relation \eqref{eqthreeterm} satisfies for $k\in\{2,3,\ldots,q\}$ the relation
\[\rho_k(x)=\frac{2(p-q+k)}{x}-\frac{1}{\rho_{k-1}(x)}.\]
On the other hand, by using again \eqref{eqthreeterm}, we have
\[\rho_1(j_{p\pm q,n})=\frac{J_{p\pm q+2}(j_{p\pm q,n})}{J_{p\pm q+1}(j_{p\pm q,n})}=\frac{2(p\pm q+1)}{j_{p\pm q,n}}\]
and this in turn implies that
\[\beta_0^{(\pm q)}=D_p\left(j_{p\pm q,n}^{2}\right)=\frac{\rho_q(j_{p\pm q,n})}{2j_{p\pm q,n}}.\]
In conclusion it is possible to apply the approach of the previous section for the more general Calogero type series, however we need first to find recursively the initial value $\beta_0^{(\pm q)}.$

\section*{\bf Appendix A: the series $\mathcal{S}_r(p,n)$ up to degree $r=15$}	

\begin{flalign*}
&\sum_{m\geq1}\frac{j_{p,n}^{2}j_{p+1,m}^{2}}{(j_{p+1,m}^{2}-j_{p,n}^{2})^{2}}
= \frac{j_{p,n}^{2}}{4}&&
\end{flalign*}
	
\begin{flalign*}
&\sum_{m\geq1}\frac{j_{p,n}^{2}j_{p+1,m}^{2}}{(j_{p+1,m}^{2}-j_{p,n}^{2})^{3}}
= \frac{p+1}{8}&&
\end{flalign*}
	
\begin{flalign*}
&\sum_{m\geq1}\frac{j_{p,n}^{2}j_{p+1,m}^{2}}{(j_{p+1,m}^{2}-j_{p,n}^{2})^{4}}= \frac{j_{p,n}^{2} {}+ 2 p^{2} {}- 2}{48 j_{p,n}^{2}}&&
\end{flalign*}
	
\begin{flalign*}
&\sum_{m\geq1}\frac{j_{p,n}^{2}j_{p+1,m}^{2}}{(j_{p+1,m}^{2}-j_{p,n}^{2})^{5}}= \frac{4 j_{p,n}^{2} p {}+ j_{p,n}^{2} {}+ 2 p^{3} {}- 4 p^{2} {}- 2 p {}+ 4}{192 j_{p,n}^{4}}&&
\end{flalign*}
	
\begin{flalign*}
\sum_{m\geq1}\frac{j_{p,n}^{2}j_{p+1,m}^{2}}{(j_{p+1,m}^{2}-j_{p,n}^{2})^{6}}&= \frac{1}{960 j_{p,n}^{6}}\left(2 j_{p,n}^{4} {}+ 11 j_{p,n}^{2} p^{2} {}- 5 	j_{p,n}^{2} p {}- 4 j_{p,n}^{2} {}+ 2 p^{4} {}- 10 p^{3}\right.\\
		&\left.\quad {}+ 10 p^{2} {}+ 10 p {}- 12\right)&&
\end{flalign*}
	
\begin{flalign*}
\sum_{m\geq1}\frac{j_{p,n}^{2}j_{p+1,m}^{2}}{(j_{p+1,m}^{2}-j_{p,n}^{2})^{7}}&= \frac{1}{11520 j_{p,n}^{8}}\left(34 j_{p,n}^{4} p {}- j_{p,n}^{4} {}+ 52 j_{p,n}^{2} p^{3} {}- 98 j_{p,n}^{2} p^{2} {}+ 2 j_{p,n}^{2} p\right.\\
		&\left.\quad {}+ 32 j_{p,n}^{2} {}+ 4 p^{5} {}- 36 p^{4} {}+ 100 p^{3} {}- 60 p^{2} {}- 104 p {}+ 96\right)&&
\end{flalign*}
	
\begin{flalign*}
\sum_{m\geq1}\frac{j_{p,n}^{2}j_{p+1,m}^{2}}{(j_{p+1,m}^{2}-j_{p,n}^{2})^{8}}&= \frac{1}{80640 j_{p,n}^{10}}\left(17 j_{p,n}^{6} {}+ 180 j_{p,n}^{4} p^{2} {}- 126 j_{p,n}^{4} p {}- 24 j_{p,n}^{4} {}+ 114 j_{p,n}^{2} p^{4}\right.\\
		&\quad {}- 448 j_{p,n}^{2} p^{3} {}+ 422 j_{p,n}^{2} p^{2} {}+ 112 j_{p,n}^{2} p {}- 152 j_{p,n}^{2} {}+ 4 p^{6}\\
		&\quad {}- 56 p^{5} {}+ 280 p^{4} {}- 560 p^{3} {}+ 196 p^{2} {}+ 616 p {}- 480)&&
\end{flalign*}
	
\begin{flalign*}
\sum_{m\geq1}\frac{j_{p,n}^{2}j_{p+1,m}^{2}}{(j_{p+1,m}^{2}-j_{p,n}^{2})^{9}}&= \frac{1}{645120 j_{p,n}^{12}}\left(248 j_{p,n}^{6} p {}- 46 j_{p,n}^{6} {}+ 768 j_{p,n}^{4} p^{3} {}- 1458 j_{p,n}^{4} p^{2}\right.\\
		&\quad {}+ 375 j_{p,n}^{4} p {}+ 207 j_{p,n}^{4} {}+ 240 j_{p,n}^{2} p^{5} {}- 1566 j_{p,n}^{2} p^{4}\\
		&\quad {}+ 3320 j_{p,n}^{2} p^{3} {}- 1930 j_{p,n}^{2} p^{2} {}- 1160 j_{p,n}^{2} p {}+ 856 j_{p,n}^{2} {}+ 4 p^{7}\\
		&\quad {}- 80 p^{6} {}+ 616 p^{5} {}- 2240 p^{4} {}+ 3556 p^{3} {}- 560 p^{2} {}- 4176 p {}+ 2880)&&
\end{flalign*}
	
\begin{flalign*}
\sum_{m\geq1}\frac{j_{p,n}^{2}j_{p+1,m}^{2}}{(j_{p+1,m}^{2}-j_{p,n}^{2})^{10}}&= \frac{1}{11612160 j_{p,n}^{14}}\left(248 j_{p,n}^{8} {}+ 4288 j_{p,n}^{6} p^{2} {}- 3648 j_{p,n}^{6} p {}+ 11 j_{p,n}^{6}\right.\\
		&\quad {}+ 5808 j_{p,n}^{4} p^{4} {}- 20640 j_{p,n}^{4} p^{3} {}+ 19890 j_{p,n}^{4} p^{2} {}- 888 j_{p,n}^{4} p\\
		&\quad {}- 3162 j_{p,n}^{4} {}+ 988 j_{p,n}^{2} p^{6} {}- 9516 j_{p,n}^{2} p^{5} {}+ 33772 j_{p,n}^{2} p^{4}\\
		&\quad {}- 50340 j_{p,n}^{2} p^{3} {}+ 18232 j_{p,n}^{2} p^{2} {}+ 20976 j_{p,n}^{2} p {}- 11232 j_{p,n}^{2}\\
		&\quad {}+ 8 p^{8} {}- 216 p^{7} {}+ 2352 p^{6} {}- 13104 p^{5} {}+ 38472 p^{4} {}- 50904 p^{3} {}- 512 p^{2}\\
		&\quad {}+ 64224 p {}- 40320)&&
\end{flalign*}
	
\begin{flalign*}
\sum_{m\geq1}\frac{j_{p,n}^{2}j_{p+1,m}^{2}}{(j_{p+1,m}^{2}-j_{p,n}^{2})^{11}}&= \frac{1}{116121600 j_{p,n}^{16}}\left(5528 j_{p,n}^{8} p {}- 1576 j_{p,n}^{8} {}+ 28768 j_{p,n}^{6} p^{3} {}- 55472 j_{p,n}^{6} p^{2}\right.\\
		&\quad {}+ 21248 j_{p,n}^{6} p {}+ 2843 j_{p,n}^{6} {}+ 20388 j_{p,n}^{4} p^{5} {}- 114558 j_{p,n}^{4} p^{4}\\
		&\quad {}+ 216474 j_{p,n}^{4} p^{3} {}- 134634 j_{p,n}^{4} p^{2} {}- 18774 j_{p,n}^{4} p {}+ 25344 j_{p,n}^{4}\\
		&\quad {}+ 2008 j_{p,n}^{2} p^{7} {}- 26564 j_{p,n}^{2} p^{6} {}+ 138628 j_{p,n}^{2} p^{5}\\
		&\quad {}- 351764 j_{p,n}^{2} p^{4} {}+ 407092 j_{p,n}^{2} p^{3} {}- 79016 j_{p,n}^{2} p^{2} {}- 194928 j_{p,n}^{2} p\\
		&\quad {}+ 84384 j_{p,n}^{2} {}+ 8 p^{9} {}- 280 p^{8} {}+ 4080 p^{7} {}- 31920 p^{6} {}+ 143304 p^{5}\\
		&\quad {}- 358680 p^{4} {}+ 406720 p^{3} {}+ 68320 p^{2} {}- 554112 p {}+ 322560)&&
\end{flalign*}
	
\begin{flalign*}
\sum_{m\geq1}\frac{j_{p,n}^{2}j_{p+1,m}^{2}}{(j_{p+1,m}^{2}-j_{p,n}^{2})^{12}}&= \frac{1}{1277337600 j_{p,n}^{18}}\left(2764 j_{p,n}^{10} {}+ 70792 j_{p,n}^{8} p^{2} {}- 67408 j_{p,n}^{8} p {}+ 6568 j_{p,n}^{8}\right.\\
		&\quad {}+ 166042 j_{p,n}^{6} p^{4} {}- 559724 j_{p,n}^{6} p^{3} {}+ 552983 j_{p,n}^{6} p^{2}\\
		&\quad {}- 106579 j_{p,n}^{6} p {}- 39414 j_{p,n}^{6} {}+ 68192 j_{p,n}^{4} p^{6} {}- 549780 j_{p,n}^{4} p^{5}\\
		&\quad {}+ 1656230 j_{p,n}^{4} p^{4} {}- 2179650 j_{p,n}^{4} p^{3} {}+ 926978 j_{p,n}^{4} p^{2}\\
		&\quad {}+ 337590 j_{p,n}^{4} p {}- 220680 j_{p,n}^{4} {}+ 4052 j_{p,n}^{2} p^{8} {}- 70136 j_{p,n}^{2} p^{7}\\
		&\quad {}+ 499624 j_{p,n}^{2} p^{6} {}- 1864016 j_{p,n}^{2} p^{5} {}+ 3763060 j_{p,n}^{2} p^{4}\\
		&\quad {}- 3547544 j_{p,n}^{2} p^{3} {}+ 158096 j_{p,n}^{2} p^{2} {}+ 1933536 j_{p,n}^{2} p {}- 715392 j_{p,n}^{2}\\
		&\quad {}+ 8 p^{10} {}- 352 p^{9} {}+ 6600 p^{8} {}- 68640 p^{7} {}+ 430584 p^{6} {}- 1648416 p^{5} {}+ 3634840 p^{4}\\
		&\quad {}- 3592160 p^{3} {}- 1168992 p^{2} {}+ 5309568 p {}- 2903040)&&
\end{flalign*}
	
\begin{flalign*}
\sum_{m\geq1}\frac{j_{p,n}^{2}j_{p+1,m}^{2}}{(j_{p+1,m}^{2}-j_{p,n}^{2})^{13}}&= \frac{1}{30656102400 j_{p,n}^{20}}\left(174752 j_{p,n}^{10} p {}- 61880 j_{p,n}^{10} {}+ 1372088 j_{p,n}^{8} p^{3} {}- 2684932 j_{p,n}^{8} p^{2}\right.\\
		&\quad {}+ 1231870 j_{p,n}^{8} p {}- 4211 j_{p,n}^{8} {}+ 1737488 j_{p,n}^{6} p^{5} {}- 8939640 j_{p,n}^{6} p^{4}\\
		&\quad {}+ 15893948 j_{p,n}^{6} p^{3} {}- 10402014 j_{p,n}^{6} p^{2} {}+ 627350 j_{p,n}^{6} p {}+ 919356 j_{p,n}^{6}\\
		&\quad {}+ 441568 j_{p,n}^{4} p^{7} {}- 4796368 j_{p,n}^{4} p^{6} {}+ 20709664 j_{p,n}^{4} p^{5}\\
		&\quad {}- 44099608 j_{p,n}^{4} p^{4} {}+ 44768872 j_{p,n}^{4} p^{3} {}- 12649192 j_{p,n}^{4} p^{2}\\
		&\quad {}- 9182664 j_{p,n}^{4} p {}+ 4202928 j_{p,n}^{4} {}+ 16288 j_{p,n}^{2} p^{9} {}- 356216 j_{p,n}^{2} p^{8}\\
		&\quad {}+ 3299568 j_{p,n}^{2} p^{7} {}- 16738032 j_{p,n}^{2} p^{6} {}+ 49812672 j_{p,n}^{2} p^{5}\\
		&\quad {}- 84590520 j_{p,n}^{2} p^{4} {}+ 66607472 j_{p,n}^{2} p^{3} {}+ 6877472 j_{p,n}^{2} p^{2}\\
		&\quad {}- 41353920 j_{p,n}^{2} p {}+ 13522176 j_{p,n}^{2} {}+ 16 p^{11} {}- 864 p^{10} {}+ 20240 p^{9}\\
		&\quad {}- 269280 p^{8} {}+ 2233968 p^{7} {}- 11908512 p^{6} {}+ 40238000 p^{5} {}- 79881120 p^{4} {}+ 69505216 p^{3}\\
		&\quad {}+ 33998976 p^{2} {}- 111997440 p {}+ 58060800)&&
\end{flalign*}

\begin{flalign*}
\sum_{m\geq1}\frac{j_{p,n}^{2}j_{p+1,m}^{2}}{(j_{p+1,m}^{2}-j_{p,n}^{2})^{14}}&= \frac{1}{398529331200 j_{p,n}^{22}}\left(87376 j_{p,n}^{12} {}+ 3106644 j_{p,n}^{10} p^{2} {}- 3184428 j_{p,n}^{10} p {}+ 487719 j_{p,n}^{10}\right.\\
		&\quad {}+ 11204160 j_{p,n}^{8} p^{4} {}- 36591464 j_{p,n}^{8} p^{3} {}+ 36855856 j_{p,n}^{8} p^{2}\\
		&\quad {}- 10098790 j_{p,n}^{8} p {}- 918952 j_{p,n}^{8} {}+ 8495440 j_{p,n}^{6} p^{6} {}- 61388080 j_{p,n}^{6} p^{5}\\
		&\quad {}+ 167712784 j_{p,n}^{6} p^{4} {}- 206789128 j_{p,n}^{6} p^{3} {}+ 97315576 j_{p,n}^{6} p^{2}\\
		&\quad {}+ 6480032 j_{p,n}^{6} p {}- 10653168 j_{p,n}^{6} {}+ 1398000 j_{p,n}^{4} p^{8}\\
		&\quad {}- 19577376 j_{p,n}^{4} p^{7} {}+ 113466552 j_{p,n}^{4} p^{6} {}- 346916232 j_{p,n}^{4} p^{5}\\
		&\quad {}+ 582102816 j_{p,n}^{4} p^{4} {}- 479160864 j_{p,n}^{4} p^{3} {}+ 77424168 j_{p,n}^{4} p^{2}\\
		&\quad {}+ 120291912 j_{p,n}^{4} p {}- 43706736 j_{p,n}^{4} {}+ 32664 j_{p,n}^{2} p^{10} {}- 879112 j_{p,n}^{2} p^{9}\\
		&\quad {}+ 10230800 j_{p,n}^{2} p^{8} {}- 67170480 j_{p,n}^{2} p^{7} {}+ 271180632 j_{p,n}^{2} p^{6}\\
		&\quad {}- 682126536 j_{p,n}^{2} p^{5} {}+ 1007581440 j_{p,n}^{2} p^{4} {}- 670641920 j_{p,n}^{2} p^{3}\\
		&\quad {}- 175479296 j_{p,n}^{2} p^{2} {}+ 477330048 j_{p,n}^{2} p {}- 141027840 j_{p,n}^{2} {}+ 16 p^{12}\\
		&\quad {}- 1040 p^{11} {}+ 29744 p^{10} {}- 491920 p^{9} {}+ 5196048 p^{8} {}- 36482160 p^{7} {}+ 171231632 p^{6}\\
		&\quad {}- 522499120 p^{5} {}+ 948197536 p^{4} {}- 730558400 p^{3} {}- 485986176 p^{2} {}+ 1290032640 p {}- 638668800)&&
	\end{flalign*}
	
\begin{flalign*}
\sum_{m\geq1}\frac{j_{p,n}^{2}j_{p+1,m}^{2}}{(j_{p+1,m}^{2}-j_{p,n}^{2})^{15}}&= \frac{1}{5579410636800 j_{p,n}^{24}}\left(3718276 j_{p,n}^{12} p {}- 1504838 j_{p,n}^{12} {}+ 41048184 j_{p,n}^{10} p^{3}\right.\\
		&\quad {}- 81333348 j_{p,n}^{10} p^{2} {}+ 41375904 j_{p,n}^{10} p {}- 2894868 j_{p,n}^{10}\\
		&\quad {}+ 81507120 j_{p,n}^{8} p^{5} {}- 395654760 j_{p,n}^{8} p^{4} {}+ 677818984 j_{p,n}^{8} p^{3}\\
		&\quad {}- 458742962 j_{p,n}^{8} p^{2} {}+ 73913273 j_{p,n}^{8} p {}+ 18345215 j_{p,n}^{8}\\
		&\quad {}+ 39573760 j_{p,n}^{6} p^{7} {}- 379713680 j_{p,n}^{6} p^{6} {}+ 1456871200 j_{p,n}^{6} p^{5}\\
		&\quad {}- 2792390648 j_{p,n}^{6} p^{4} {}+ 2652739720 j_{p,n}^{6} p^{3} {}- 912571952 j_{p,n}^{6} p^{2}\\
		&\quad {}- 202425840 j_{p,n}^{6} p {}+ 128296656 j_{p,n}^{6} {}+ 4357320 j_{p,n}^{4} p^{9}\\
		&\quad {}- 76155060 j_{p,n}^{4} p^{8} {}+ 566612856 j_{p,n}^{4} p^{7} {}- 2320772424 j_{p,n}^{4} p^{6}\\
		&\quad {}+ 5614074072 j_{p,n}^{4} p^{5} {}- 7856930436 j_{p,n}^{4} p^{4} {}+ 5382434424 j_{p,n}^{4} p^{3}\\
		&\quad {}- 249841056 j_{p,n}^{4} p^{2} {}- 1610379792 j_{p,n}^{4} p {}+ 494345376 j_{p,n}^{4}\\
		&\quad {}+ 65424 j_{p,n}^{2} p^{11} {}- 2122632 j_{p,n}^{2} p^{10} {}+ 30249032 j_{p,n}^{2} p^{9}\\
		&\quad {}- 248398480 j_{p,n}^{2} p^{8} {}+ 1293390912 j_{p,n}^{2} p^{7} {}- 4408770696 j_{p,n}^{2} p^{6}\\
		&\quad {}+ 9710360616 j_{p,n}^{2} p^{5} {}- 12752398560 j_{p,n}^{2} p^{4} {}+ 7195296064 j_{p,n}^{2} p^{3}\\
		&\quad {}+ 3188614528 j_{p,n}^{2} p^{2} {}- 5934987648 j_{p,n}^{2} p {}+ 1609367040 j_{p,n}^{2} {}+ 16 p^{13}\\
		&\quad {}- 1232 p^{12} {}+ 42224 p^{11} {}- 848848 p^{10} {}+ 11099088 p^{9} {}- 98834736 p^{8} {}+ 609017552 p^{7}\\
		&\quad {}- 2577278704 p^{6} {}+ 7218186976 p^{5} {}- 12108928832 p^{4} {}+ 8280714624 p^{3} {}+ 7121866752 p^{2}\\
		&\quad {}- 16119060480 p {}+ 7664025600)&&
\end{flalign*}
	
\section*{\bf Appendix B: the series $\Sigma_r^{(+1)}(p,n)$ up to degree $r=10$}

\begin{flalign*}
\sum_{m\geq1}
\frac{1}{\left(j_{p,m}^2-j_{p+1,n}^2\right)^{1}}
= 0&&
\end{flalign*}

\begin{flalign*}
\sum_{m\geq1}
\frac{1}{\left(j_{p,m}^2-j_{p+1,n}^2\right)^{2}}
= \frac{1}{4 j_{p+1,n}^{2}}&&
\end{flalign*}

\begin{flalign*}
\sum_{m\geq1}
\frac{1}{\left(j_{p,m}^2-j_{p+1,n}^2\right)^{3}}
= - \frac{p + 2}{8 j_{p+1,n}^{4}}&&
\end{flalign*}

\begin{flalign*}
\sum_{m\geq1}
\frac{1}{\left(j_{p,m}^2-j_{p+1,n}^2\right)^{4}}
= \frac{j_{p+1,n}^{2} + 2 p^{2} + 10 p + 12}{48 j_{p+1,n}^{6}}&&
\end{flalign*}

\begin{flalign*}
\sum_{m\geq1}
\frac{1}{\left(j_{p,m}^2-j_{p+1,n}^2\right)^{5}}
&= - \frac{1}{192 j_{p+1,n}^{8}}\left(4 j_{p+1,n}^{2} p + 7 j_{p+1,n}^{2} + 2 p^{3} + 18 p^{2}\right.\\
&\left.\quad {}+ 52 p + 48\right)&&
\end{flalign*}

\begin{flalign*}
\sum_{m\geq1}
\frac{1}{\left(j_{p,m}^2-j_{p+1,n}^2\right)^{6}}
&= \frac{1}{960 j_{p+1,n}^{10}}\left(2 j_{p+1,n}^{4} + 11 j_{p+1,n}^{2} p^{2} + 47 j_{p+1,n}^{2} p\right.\\
&\quad {}+ 47 j_{p+1,n}^{2} + 2 p^{4} + 28 p^{3} + 142 p^{2} + 308 p + 240)&&
\end{flalign*}

\begin{flalign*}
\sum_{m\geq1}
\frac{1}{\left(j_{p,m}^2-j_{p+1,n}^2\right)^{7}}
&= - \frac{1}{11520 j_{p+1,n}^{12}}\left(34 j_{p+1,n}^{4} p + 59 j_{p+1,n}^{4} + 52 j_{p+1,n}^{2} p^{3}\right.\\
&\quad {}+ 386 j_{p+1,n}^{2} p^{2} + 918 j_{p+1,n}^{2} p + 684 j_{p+1,n}^{2}\\
&\quad {}+ 4 p^{5} + 80 p^{4} + 620 p^{3} + 2320 p^{2} + 4176 p + 2880)&&
\end{flalign*}

\begin{flalign*}
\sum_{m\geq1}
\frac{1}{\left(j_{p,m}^2-j_{p+1,n}^2\right)^{8}}
&= \frac{1}{80640 j_{p+1,n}^{14}}\left(17 j_{p+1,n}^{6} + 180 j_{p+1,n}^{4} p^{2} + 724 j_{p+1,n}^{4} p\right.\\
&\quad {}+ 695 j_{p+1,n}^{4} + 114 j_{p+1,n}^{2} p^{4} + 1268 j_{p+1,n}^{2} p^{3}\\
&\quad {}+ 5152 j_{p+1,n}^{2} p^{2} + 8958 j_{p+1,n}^{2} p + 5508 j_{p+1,n}^{2}\\
&\quad {}+ 4 p^{6} + 108 p^{5} + 1180 p^{4} + 6660 p^{3}\\
&\quad {}+ 20416 p^{2} + 32112 p + 20160)&&
\end{flalign*}

\begin{flalign*}
\sum_{m\geq1}
\frac{1}{\left(j_{p,m}^2-j_{p+1,n}^2\right)^{9}}
&= - \frac{1}{645120 j_{p+1,n}^{16}}\left(248 j_{p+1,n}^{6} p + 430 j_{p+1,n}^{6} + 768 j_{p+1,n}^{4} p^{3}\right.\\
&\quad {}+ 5202 j_{p+1,n}^{4} p^{2} + 11387 j_{p+1,n}^{4} p + 7954 j_{p+1,n}^{4}\\
&\quad {}+ 240 j_{p+1,n}^{2} p^{5} + 3678 j_{p+1,n}^{2} p^{4} + 22128 j_{p+1,n}^{2} p^{3}\\
&\quad {}+ 64902 j_{p+1,n}^{2} p^{2} + 91788 j_{p+1,n}^{2} p + 49104 j_{p+1,n}^{2}\\
&\quad {}+ 4 p^{7} + 140 p^{6} + 2044 p^{5} + 16100 p^{4}\\
&\quad {}+ 73696 p^{3} + 195440 p^{2} + 277056 p + 161280)&&
\end{flalign*}

\begin{flalign*}
\sum_{m\geq1}
\frac{1}{\left(j_{p,m}^2-j_{p+1,n}^2\right)^{10}}
&= \frac{1}{11612160 j_{p+1,n}^{18}}\left(248 j_{p+1,n}^{8} + 4288 j_{p+1,n}^{6} p^{2} + 16688 j_{p+1,n}^{6} p\right.\\
&\quad {}+ 15687 j_{p+1,n}^{6} + 5808 j_{p+1,n}^{4} p^{4} + 57696 j_{p+1,n}^{4} p^{3}\\
&\quad {}+ 210294 j_{p+1,n}^{4} p^{2} + 330786 j_{p+1,n}^{4} p + 187236 j_{p+1,n}^{4}\\
&\quad {}+ 988 j_{p+1,n}^{2} p^{6} + 19764 j_{p+1,n}^{2} p^{5} + 162376 j_{p+1,n}^{2} p^{4}\\
&\quad {}+ 698652 j_{p+1,n}^{2} p^{3} + 1650100 j_{p+1,n}^{2} p^{2}\\
&\quad {}+ 2007288 j_{p+1,n}^{2} p + 964512 j_{p+1,n}^{2} + 8 p^{8} + 352 p^{7}\\
&\quad {}+ 6608 p^{6} + 68992 p^{5} + 437192 p^{4} + 1717408 p^{3}\\
&\quad {}+ 4072032 p^{2} + 5309568 p + 2903040)&&
\end{flalign*}

\section*{\bf Appendix C: the series $\Sigma_r^{(-1)}(p,n)$ up to degree $r=10$}

\begin{flalign*}
\sum_{m\geq1}
\frac{1}{\left(j_{p,m}^2-j_{p-1,n}^2\right)^{1}}
= \frac{p}{j_{p-1,n}^{2}}&&
\end{flalign*}

\begin{flalign*}
\sum_{m\geq1}
\frac{1}{\left(j_{p,m}^2-j_{p-1,n}^2\right)^{2}}
= - \frac{- j_{p-1,n}^{2} + 4 p}{4 j_{p-1,n}^{4}}&&
\end{flalign*}

\begin{flalign*}
\sum_{m\geq1}
\frac{1}{\left(j_{p,m}^2-j_{p-1,n}^2\right)^{3}}
= \frac{j_{p-1,n}^{2} p - 2 j_{p-1,n}^{2} + 8 p}{8 j_{p-1,n}^{6}}&&
\end{flalign*}

\begin{flalign*}
\sum_{m\geq1}
\frac{1}{\left(j_{p,m}^2-j_{p-1,n}^2\right)^{4}}
&= \frac{j_{p-1,n}^{4} + 2 j_{p-1,n}^{2} p^{2} - 10 j_{p-1,n}^{2} p + 12 j_{p-1,n}^{2} - 48 p}{48 j_{p-1,n}^{8}}&&
\end{flalign*}

\begin{flalign*}
\sum_{m\geq1}
\frac{1}{\left(j_{p,m}^2-j_{p-1,n}^2\right)^{5}}
&= \frac{1}{192 j_{p-1,n}^{10}}\left(4 j_{p-1,n}^{4} p - 7 j_{p-1,n}^{4} + 2 j_{p-1,n}^{2} p^{3}\right.\\
&\quad {}- 18 j_{p-1,n}^{2} p^{2} + 52 j_{p-1,n}^{2} p - 48 j_{p-1,n}^{2} + 192 p)&&
\end{flalign*}

\begin{flalign*}
\sum_{m\geq1}
\frac{1}{\left(j_{p,m}^2-j_{p-1,n}^2\right)^{6}}
&= \frac{1}{960 j_{p-1,n}^{12}}\left(2 j_{p-1,n}^{6} + 11 j_{p-1,n}^{4} p^{2} - 47 j_{p-1,n}^{4} p\right.\\
&\quad {}+ 47 j_{p-1,n}^{4} + 2 j_{p-1,n}^{2} p^{4} - 28 j_{p-1,n}^{2} p^{3}\\
&\quad {}+ 142 j_{p-1,n}^{2} p^{2} - 308 j_{p-1,n}^{2} p + 240 j_{p-1,n}^{2} - 960 p)&&
\end{flalign*}

\begin{flalign*}
\sum_{m\geq1}
\frac{1}{\left(j_{p,m}^2-j_{p-1,n}^2\right)^{7}}
&= \frac{1}{11520 j_{p-1,n}^{14}}\left(34 j_{p-1,n}^{6} p - 59 j_{p-1,n}^{6} + 52 j_{p-1,n}^{4} p^{3}\right.\\
&\quad {}- 386 j_{p-1,n}^{4} p^{2} + 918 j_{p-1,n}^{4} p - 684 j_{p-1,n}^{4}\\
&\quad {}+ 4 j_{p-1,n}^{2} p^{5} - 80 j_{p-1,n}^{2} p^{4} + 620 j_{p-1,n}^{2} p^{3}\\
&\quad {}- 2320 j_{p-1,n}^{2} p^{2} + 4176 j_{p-1,n}^{2} p - 2880 j_{p-1,n}^{2}\\
&\quad {}+ 11520 p)&&
\end{flalign*}

\begin{flalign*}
\sum_{m\geq1}
\frac{1}{\left(j_{p,m}^2-j_{p-1,n}^2\right)^{8}}
&= \frac{1}{80640 j_{p-1,n}^{16}}\left(17 j_{p-1,n}^{8} + 180 j_{p-1,n}^{6} p^{2} - 724 j_{p-1,n}^{6} p\right.\\
&\quad {}+ 695 j_{p-1,n}^{6} + 114 j_{p-1,n}^{4} p^{4} - 1268 j_{p-1,n}^{4} p^{3}\\
&\quad {}+ 5152 j_{p-1,n}^{4} p^{2} - 8958 j_{p-1,n}^{4} p + 5508 j_{p-1,n}^{4}\\
&\quad {}+ 4 j_{p-1,n}^{2} p^{6} - 108 j_{p-1,n}^{2} p^{5} + 1180 j_{p-1,n}^{2} p^{4}\\
&\quad {}- 6660 j_{p-1,n}^{2} p^{3} + 20416 j_{p-1,n}^{2} p^{2} - 32112 j_{p-1,n}^{2} p\\
&\quad {}+ 20160 j_{p-1,n}^{2} - 80640 p)&&
\end{flalign*}

\begin{flalign*}
\sum_{m\geq1}
\frac{1}{\left(j_{p,m}^2-j_{p-1,n}^2\right)^{9}}
&= \frac{1}{645120 j_{p-1,n}^{18}}\left(248 j_{p-1,n}^{8} p - 430 j_{p-1,n}^{8} + 768 j_{p-1,n}^{6} p^{3}\right.\\
&\quad {}- 5202 j_{p-1,n}^{6} p^{2} + 11387 j_{p-1,n}^{6} p - 7954 j_{p-1,n}^{6}\\
&\quad {}+ 240 j_{p-1,n}^{4} p^{5} - 3678 j_{p-1,n}^{4} p^{4} + 22128 j_{p-1,n}^{4} p^{3}\\
&\quad {}- 64902 j_{p-1,n}^{4} p^{2} + 91788 j_{p-1,n}^{4} p - 49104 j_{p-1,n}^{4}\\
&\quad {}+ 4 j_{p-1,n}^{2} p^{7} - 140 j_{p-1,n}^{2} p^{6} + 2044 j_{p-1,n}^{2} p^{5}\\
&\quad {}- 16100 j_{p-1,n}^{2} p^{4} + 73696 j_{p-1,n}^{2} p^{3}\\
&\quad {}- 195440 j_{p-1,n}^{2} p^{2} + 277056 j_{p-1,n}^{2} p - 161280 j_{p-1,n}^{2}\\
&\quad {}+ 645120 p)&&
\end{flalign*}

\begin{flalign*}
\sum_{m\geq1}
\frac{1}{\left(j_{p,m}^2-j_{p-1,n}^2\right)^{10}}
&= \frac{1}{11612160 j_{p-1,n}^{20}}\left(248 j_{p-1,n}^{10} + 4288 j_{p-1,n}^{8} p^{2} - 16688 j_{p-1,n}^{8} p\right.\\
&\quad {}+ 15687 j_{p-1,n}^{8} + 5808 j_{p-1,n}^{6} p^{4} - 57696 j_{p-1,n}^{6} p^{3}\\
&\quad {}+ 210294 j_{p-1,n}^{6} p^{2} - 330786 j_{p-1,n}^{6} p + 187236 j_{p-1,n}^{6}\\
&\quad {}+ 988 j_{p-1,n}^{4} p^{6} - 19764 j_{p-1,n}^{4} p^{5} + 162376 j_{p-1,n}^{4} p^{4}\\
&\quad {}- 698652 j_{p-1,n}^{4} p^{3} + 1650100 j_{p-1,n}^{4} p^{2}\\
&\quad {}- 2007288 j_{p-1,n}^{4} p + 964512 j_{p-1,n}^{4} + 8 j_{p-1,n}^{2} p^{8}\\
&\quad {}- 352 j_{p-1,n}^{2} p^{7} + 6608 j_{p-1,n}^{2} p^{6} - 68992 j_{p-1,n}^{2} p^{5}\\
&\quad {}+ 437192 j_{p-1,n}^{2} p^{4} - 1717408 j_{p-1,n}^{2} p^{3}\\
&\quad {}+ 4072032 j_{p-1,n}^{2} p^{2} - 5309568 j_{p-1,n}^{2} p + 2903040 j_{p-1,n}^{2}\\
&\quad {}- 11612160 p)&&
\end{flalign*}
	
\end{document}